\documentclass[12pt]{amsart}
\usepackage{amsmath,amssymb,amsthm,amscd,amsfonts,geometry,color}
\renewcommand{\labelenumi}{\rm(\theenumi)}

\theoremstyle{plain}

\newtheorem{theorem}{Theorem}
\newtheorem{proposition}{Proposition}
\newtheorem{lemma}{Lemma}

\newtheorem{main}{Main Theorem}

\theoremstyle{remark}

\theoremstyle{definition}

\newcommand{\R}{\mathbb{R}}                       % real numbers
\begin{document}

\title[The topological classification of spaces of metrics]{The topological classification of spaces of metrics with the uniform convergence topology}
\author{Katsuhisa Koshino}
\address[Katsuhisa Koshino]{Faculty of Engineering, Kanagawa University, Yokohama, 221-8686, Japan}
\email{ft160229no@kanagawa-u.ac.jp}
\maketitle

\begin{abstract}
For a metrizable space $X$ of density $\kappa$, let $PM(X)$ be the space of continuous bounded pseudometrics on $X$ endowed with the uniform convergence topology.
In this paper, its topology shall be classified as follows:
(i)~If $X$ is finite,
 then $PM(X)$ is homeomorphic to $\{0\}$ when $X$ is a singleton,
 and then $PM(X)$ is homeomorphic to $[0,1]^{\kappa(\kappa - 1)/2 - 1} \times [0,1)$ when $\kappa > 1$;
(ii)~If $X$ is infinite and generalized compact,
 then $PM(X)$ is homeomorphic to the Hilbert space $\ell_2(2^{< \kappa})$ of density $2^{< \kappa}$;
(iii)~If $X$ is not generalized compact,
 then $PM(X)$ is homeomorphic to the Hilbert space $\ell_2(2^\kappa)$ of density $2^\kappa$.
Furthermore, letting $M(X)$ and $AM(X)$ be the spaces of continuous bounded metrics and bounded admissible metrics on $X$ with the subspace topology of $PM(X)$ respectively, we will recognize their topological types as follows:
(iv)~If $X$ is infinite and compact,
 then $M(X) (= AM(X))$ is homeomorphic to the separable Hilbert space $\ell_2$;
(v)~In the case that $X$ is not compact, $M(X)$ is homeomorphic to the Hilbert space $\ell_2(2^{\aleph_0})$ if $X$ is $\sigma$-compact,
 and moreover $AM(X)$ is also homeomorphic to the Hilbert space $\ell_2(2^{\aleph_0})$ if $X$ is separable locally compact.
\end{abstract}

\section*{}

Given a metrizable space $X$, we denote by $C(X^2)$ the space of continuous bounded real-valued functions on $X^2$ equipped with the uniform convergence topology.
The space $C(X^2)$ is a Banach space with the sup-norm,
 and then the sup-metric $D$ is admissible on $C(X^2)$:
 $$D(f,g) = \sup\{|f(x,y) - g(x,y)| \mid (x,y) \in X^2\}$$
 for any $f, g \in C(X^2)$.
Let $PM(X)$, $M(X)$ and $AM(X)$ be the spaces consisting of continuous bounded pseudometrics, continuous bounded metrics and bounded admissible metrics on $X$ with the subspace topology of $C(X^2)$, respectively.
As is easily observed,
 $PM(X)$ is a convex non-negative cone,
 and $M(X)$ and $AM(X)$ are convex positive cones, in the linear space $C(X^2)$.
Note that $M(X)$ coincides with $AM(X)$ when $X$ is compact.
In this paper, we shall classify the topology of $PM(X)$ as follows:

\begin{main}
Let $X$ be a metrizable space of density $\kappa$.
\begin{enumerate}
 \renewcommand{\labelenumi}{(\roman{enumi})}
 \item If $X$ is finite,
 then $PM(X)$ is homeomorphic to $\{0\}$ when $X$ is a singleton,
 and then $PM(X)$ is homeomorphic to $[0,1]^{\kappa(\kappa - 1)/2 - 1} \times [0,1)$ when $X$ is not a singleton;
 \item If $X$ is infinite and generalized compact,
 then $PM(X)$ is homeomorphic to the Hilbert space $\ell_2(2^{< \kappa})$ of density $2^{< \kappa}$;
 \item If $X$ is not generalized compact,
 then $PM(X)$ is homeomorphic to the Hilbert space $\ell_2(2^\kappa)$ of density $2^\kappa$.
\end{enumerate}
In particular, in the case where $X$ is separable, $PM(X)$ is homeomorphic to the separable Hilbert space $\ell_2$ when $X$ is infinite and compact,
 and $PM(X)$ is homeomorphic to the Hilbert space $\ell_2(2^{\aleph_0})$ when $X$ is not compact.
\end{main}

Furthermore, we will investigate the topological types of $M(X)$ and $AM(X)$ as follows:

\begin{theorem}\label{metric}
If $X$ is an infinite and compact metrizable space,
 then $M(X) (= AM(X))$ is homeomorphic to the separable Hilbert space $\ell_2$.
In the case that $X$ is not compact, $M(X)$ is homeomorphic to the Hilbert space $\ell_2(2^{\aleph_0})$ when $X$ is $\sigma$-compact,
 and moreover $AM(X)$ is also homeomorphic to the Hilbert space $\ell_2(2^{\aleph_0})$ when $X$ is separable locally compact.
\end{theorem}

A space $X$ is \emph{generalized compact} if every open cover of $X$ has a subcover whose cardinality is less than the density of $X$.
Recall that $PM(X)$ is closed in $C(X^2)$.
If $X$ is a $\sigma$-compact metrizable space,
 then $M(X)$ is $G_\delta$ in $C(X^2)$,
 and more if it is a separable locally compact metrizable space,
 then $AM(X)$ is also $G_\delta$, refer to the proof of \cite[Lemma~5.1]{Ish}.
T.~Dobrowolski and H.~Toru\'nczyk \cite[Theorem~2]{DT2}, and T.~Banakh and R.~Cauty \cite[Theorem~2]{BanC} gave a sufficient condition for $G_\delta$ convex sets in Fr\'{e}chet spaces to be homeomorphic to Hilbert spaces as follows:

\begin{theorem}\label{convex}
Let $C$ be a $G_\delta$ convex set in a Fr\'{e}chet space.
If the closure of $C$ is not locally compact,
 then $C$ is homeomorphic to a Hilbert space.
Especially, if $C$ is not separable,
 then $C$ is homeomorphic to a non-separable Hilbert space.
\end{theorem}

Using this criterion, we will prove Main Theorem and Theorem~\ref{metric}.
It is shown that $PM(X)$ is the closure of $M(X)$ and $AM(X)$ in $C(X^2)$.

\begin{proposition}\label{dense}
For every metrizable space $X$, $AM(X)$ is dense in $PM(X)$.
\end{proposition}

\begin{proof}
For each pseudometric $d \in PM(X)$ and each positive number $\epsilon$, we will find an admissible metric $\rho \in AM(X)$ with $D(d,\rho) \leq \epsilon$.
Fix any admissible metric $d_X \in AM(X)$ so that $d_X(x,y) \leq \epsilon$ for all $x, y \in X$.
Define a continuous function $\rho : X^2 \to \R$ by
 $$\rho(x,y) = \max\{d(x,y),d_X(x,y)\}$$
 for each $(x,y) \in X^2$.
It is easy to show that $\rho \in AM(X)$.
Moreover, for every pair $(x,y) \in X^2$,
 when $d(x,y) \leq \epsilon$,
 $$|d(x,y) - \rho(x,y)| = \max\{d(x,y),d_X(x,y)\} - d(x,y) \leq \max\{d(x,y),d_X(x,y)\} \leq \epsilon,$$
 and when $d(x,y) \geq \epsilon$,
 $$|d(x,y) - \rho(x,y)| = \max\{d(x,y),d_X(x,y)\} - d(x,y) = d(x,y) - d(x,y) = 0.$$
Hence $D(d,\rho) \leq \epsilon$.
We conclude that $AM(X)$ is dense in $PM(X)$.
\end{proof}

We shall estimate the densities of $PM(X)$, $M(X)$ and $AM(X)$.

\begin{lemma}\label{extent}
Let $X$ be a metrizable space.
If $X$ contains a closed discrete subset of cardinality $\kappa$,
 then the densities of $PM(X)$, $M(X)$ and $AM(X)$ are greater than or equal to $2^\kappa$.
\end{lemma}

\begin{proof}
Take a closed discrete subset $D = \{x_\gamma \in X \mid \gamma < \kappa\}$ and fix a point $x_0 \in D$.
We can define a pseudometric $d_a \in PM(D)$ for each sequence $a = (a_\gamma)_{\gamma < \kappa} \in \{0,1\}^\kappa$ by
\begin{enumerate}
 \item $d_a(x_0,x_0) = 0$ and $d_a(x_\gamma,x_0) = d_a(x_0,x_\gamma) = a_\gamma$ for any $0 < \gamma < \kappa$,
 \item $d_a(x_\gamma,x_\lambda) = |d_a(x_\gamma,x_0) - d_a(x_\lambda,x_0)|$ for any $0 < \gamma \leq \lambda < \kappa$.
\end{enumerate}
Note that the subset
 $$\{d_a \in PM(D) \mid a \in \{0,1\}^\kappa\}$$
 is closed discrete in $PM(D)$.
The density of $PM(D)$ is greater than or equal to $2^\kappa$,
 and hence so is the one of $AM(D)$ because of Proposition~\ref{dense}.
By Hausdorff's metric extension theorem \cite{Hau}, the restriction
 $$AM(X) \ni d \mapsto d|_{D^2} \in AM(D)$$
 is surjective,
 which implies that the density of $AM(X)$ is also greater than or equal to $2^\kappa$.
\end{proof}

\begin{proposition}\label{density}
Suppose that $X$ is a metrizable space of density $\kappa$.
If $X$ is infinite and generalized compact,
 then the densities of $PM(X)$, $M(X)$ and $AM(X)$ are equal to $2^{< \kappa}$,
 and if $X$ is not generalized compact,
 then their densities are equal to $2^\kappa$.
\end{proposition}

\begin{proof}
In the case that $X$ is generalized compact, by virtue of Corollary~5 and Theorem~8 of \cite{BarC}, $\kappa$ has countable cofinality.
Hence we can write $\kappa = \sup_{n < \aleph_0} \kappa_n$ with $\kappa_n < \kappa$.
Then for each $n < \aleph_0$, there exists a closed discrete set in $X$ of cardinality $\geq \kappa_n$.
Indeed, suppose that any closed discrete subset of $X$ is of cardinality $< \kappa_n$,
 so the density of $X$, which is equal to the extent,
 is less than or equal to $\kappa_n$.
This is a contradiction.
By Lemma~\ref{extent}, the densities of $PM(X)$, $M(X)$ and $AM(X)$ are greater than or equal to $2^{< \kappa}$.
In the case that $X$ is not generalized compact, there exists a closed discrete set in $X$ of cardinality $\kappa$ due to Proposition~2.3 of \cite{Cos}.
Using Lemma~\ref{extent} again, we have that their densities are greater than or equal to $2^\kappa$.
On the other hand, $C(X^2)$ is of density $2^{< \kappa}$ when $X$ is generalized compact,
 and $2^\kappa$ when $X$ is not so by \cite[Theorem~2.6]{Cos}.
Thus the proof is complete.
\end{proof}

Now Main Theorem is shown.

\begin{proof}[Proof of Main Theorem]
First, we shall prove (i).
It is clear that $PM(X)$ is homeomorphic to $\{0\}$ when $X$ is a singleton.
In the case where $\kappa > 1$,
 we show that $PM(X)$ is homeomorphic to $[0,1]^{\kappa(\kappa - 1)/2 - 1} \times [0,1)$.
Let $X = \{x_1, \cdots, x_\kappa\}$.
Since $d(x_i,x_i) = 0$ and $d(x_i,x_j) = d(x_j,x_i)$ for all $1 \leq i \leq j \leq \kappa$,
 $PM(X)$ is homeomorphic to the closed cone
 $$C = \{(d(x_i,x_j))_{1 \leq i < j \leq \kappa} \mid d \in PM(X)\} \subset [0,\infty)^{\kappa(\kappa - 1)/2}.$$
Letting
 $$H = \Bigg\{(z_{(i,j)})_{1 \leq i < j \leq \kappa} \in [0,\infty)^{\kappa(\kappa - 1)/2} \ \Bigg| \ \sum_{1 \leq i < j \leq \kappa} z_{(i,j)} = 1\Bigg\},$$
 we get that $C = [0,\infty) \cdot (C \cap H)$.
As is easily observed,
 $C \cap H$ is compact and convex.
Moreover, the dimension of $C \cap H$ is equal to $\kappa(\kappa - 1)/2 - 1$.
Indeed, $H$ is of dimension $\kappa(\kappa - 1)/2 - 1$,
 and the $(\kappa(\kappa - 1)/2 - 1)$-dimensional convex set
 $$\Bigg\{\sum_{1 \leq i < j \leq \kappa} (2s_{(i,j)}/(\kappa(\kappa - 1) -2)) \cdot w^{(i,j)} \ \Bigg| \ s_{(i,j)} \geq 0 \text{ and } \sum_{1 \leq i < j \leq \kappa} s_{(i,j)} = 1\Bigg\},$$
 where let $w^{(i,j)} = (w_{(k,l)})_{1 \leq k < l \leq \kappa} \in [0,\infty)^{\kappa(\kappa - 1)/2}$ be defined by
 $$w_{(k,l)} = \left\{
 \begin{array}{ll}
  0 & \text{ if } (k,l) = (i,j)\\
  1 & \text{ if } (k,l) \neq (i,j),
 \end{array}
 \right.$$
 is contained in $C \cap H$.
It follows that $C \cap H$ is homeomorphic to $[0,1]^{\kappa(\kappa - 1)/2 - 1}$,
 and hence $PM(X)$ is homeomorphic to $[0,1]^{\kappa(\kappa - 1)/2 - 1} \times [0,1)$.

Next, we show (ii).
Assume that $X$ is infinite and generalized compact.
To show that $PM(X)$ is not locally compact, fix any pseudometric $d \in PM(X)$ and any neighborhood $U$ of $d$.
By virtue of Arzer\`{a} - Ascoli's Theorem, we need only to prove that $U$ is not equicontinuous,
 that is, there exists a point $(x,y) \in X^2$ and a positive number $\epsilon > 0$ such that for each neighborhood $V$ of $(x,y)$,
 $$\rho(x',y') \notin (\rho(x,y) - \epsilon,\rho(x,y) + \epsilon)$$
 for some $\rho \in U$ and some $(x',y') \in V$.
Take $\epsilon > 0$ such that if $D(d,d') \leq \epsilon$,
 then $d' \in U$.
Since $X$ is infinite and generalized compact,
 we can choose a point $x \in X$ all of whose neighborhoods are of density $\geq \kappa$ according to \cite[Lemma~11]{BarC},
 which implies that $x$ is not isolated.
Hence for any neighborhood $V$ of $(x,x)$ in $X^2$, there is $y \in X$ such that $x \neq y$ and $(x,y) \in V$.
Define an admissible metric $e \in AM(\{x, y\})$ as follows:
 $$e(x,x) = 0, \ e(y,y) = 0, \ e(x,y) = e(y,x) = \epsilon.$$
Applying Hausdorff's metric extension theorem \cite{Hau}, we can obtain an admissible metric $\tilde{e} \in AM(X)$ so that $\tilde{e}|_{\{x,y\}^2} = e$ and $\tilde{e}(u,v) \leq \epsilon$ for any $(u,v) \in X^2$.
Let $\rho \in AM(X)$ be an admissible metric defined by $\rho = d + \tilde{e}$.
Then for each $(u,v) \in X^2$,
 $$|d(u,v) - \rho(u,v)| = |d(u,v) - (d(u,v) + \tilde{e}(u,v))| = \tilde{e}(u,v) \leq \epsilon.$$
Therefore we have $D(d,\rho) \leq \epsilon$,
 which implies that $\rho \in U$.
Moreover,
 $$\rho(x,y) = d(x,y) + \tilde{e}(x,y) \geq \tilde{e}(x,y) = \epsilon = \rho(x,x) + \epsilon.$$
Consequently, $U$ is not equicontinuous,
 and hence $PM(X)$ is not locally compact.
According to Proposition~\ref{density}, $PM(X)$ is of density $2^{< \kappa}$.
It follows from Theorem~\ref{convex} that $PM(X)$ is homeomorphic to $\ell_2(2^{< \kappa})$.

Combining Theorem~\ref{convex} with Proposition~\ref{density}, we can establish (iii).
The proof is complete.
\end{proof}

Recall that $PM(X)$ is the closure of $M(X)$ and $AM(X)$ by Proposition~\ref{dense}.
When $X$ is a separable metrizable space,
 the compactness of it coincides with the generalized compactness of it.
Hence as is seem in the above proof,
 $PM(X)$ is not locally compact.
Applying Theorem~\ref{convex}, we can prove Theorem~\ref{metric}.

\end{document}